\documentclass[12pt, a4paper]{amsart}
\usepackage{amsmath}
\usepackage{geometry,amsthm,graphics,tabularx,amssymb,shapepar}
\usepackage{amscd}
\usepackage[all]{xypic}

\newcommand{\Ad}{{\mathrm{Ad}}}

\theoremstyle{plain}
\newtheorem{theorem}{Theorem}[section]

\newtheorem{coro}[theorem]{Corollary}
\newtheorem{lemma}[theorem]{Lemma}
\newtheorem{ex}[theorem]{Example}

\newtheorem{rem}[theorem]{Remark}
\newtheorem{defi}[theorem]{Definition}

\newcommand{\ble}{\begin {lemma}}
\newcommand{\ele}{\end {lemma}}
\newcommand{\bprop}{\begin {Proposition}}
\newcommand{\eprop}{\end {Proposition}}
\newcommand{\bthm}{\begin {theorem}}
\newcommand{\ethm}{\end {theorem}}
\newcommand{\bco}{\begin {coro}}
\newcommand{\eco}{\end {coro}}
\newcommand{\bex}{\begin {ex}}
\newcommand{\eex}{\end {ex}}
\newcommand{\bre}{\begin {rem}}
\newcommand{\ere}{\end {rem}}

\newcommand{\be}{\begin {equation}}
\newcommand{\ee}{\end {equation}}
\newcommand{\bp}{\begin {proof}}
\newcommand{\ep}{\end {proof}}
\newcommand{\bee}{\begin {equation*}}
\newcommand{\eee}{\end {equation*}}
\newcommand{\rt}{\rightarrow}
\newcommand{\lb}{\label}

\begin{document}

\title{Real fundamental Chevalley involutions and conjugacy classes}


\author [G. Han] {Gang Han}
\address{School of Mathematics \\ Zhejiang University\\Hangzhou, 310027, China} \email{mathhgg@zju.edu.cn}

\author[B. Sun]{Binyong Sun}
\address{Institute for Advanced Study in Mathematics, Zhejiang University\\
  Hangzhou, 310058, China}\email{sunbinyong@zju.edu.cn}




\begin{abstract}

Let $\mathsf G$ be a connected reductive linear algebraic group defined over $\mathbb R$, and let $C: \mathsf G\rightarrow \mathsf G$ be a fundamental Chevalley involution. We show that for every $g\in \mathsf G(\mathbb R)$, $C(g)$ is conjugate to $g^{-1}$ in the group $\mathsf G(\mathbb R)$.  Similar result on the Lie algebras is also obtained.

\end{abstract}

 \maketitle

\section{Introduction}
 \setcounter{equation}{0}\setcounter{theorem}{0}

\newcommand{\G}{{\mathsf G}}
\newcommand{\F}{{\mathrm F}}
\newcommand{\R}{{\mathbb R}}
Let $\F$ be a field of characteristic zero, and let $\G$ be a  connected reductive linear algebraic group defined over $\F$. When  $\F$ is algebraically closed, or when $\G$ is a
general linear group, a unitary group, an orthogonal group, or a symplectic group, it is respectively shown in \cite[Proposition 2.6]{l} and
\cite[Proposition 4.I.2]{mvw} that there exists an involutive algebraic automorphism $\varphi: \G\rightarrow \G$ defined over $\F$ with the following remarkable property: for every $g \in \G(\mathrm{F})$,
\be
\lb{b1}
\text{$\varphi(g)$ and $g^{-1}$ are conjugate to each other in ${\G}(\mathrm{F})$.}
\ee
However, when $\F$ is a $p$-adic field, no such automorphism exists for some $\G$. For example, it was pointed out by D. Prasad (\cite{Pr15}, see also \cite[Example 2]{ad}) that no such automorphism exists when $\G$ is a split exceptional group of type $G_2$, $F_4$ or $E_8$, and it was 
 shown in \cite[Proposition 1.3]{lst} that no such automorphism exists when $\G$ is a quaternionic classical group of rank $\geq 5$. 
In this note, we will show that such an automorphism exists when $\F=\R$.

From now on we assume that $\F=\R$. Recall that a Cartan subgroup of $\G$ is said to be fundamental if its split rank is minimal among all Cartan subgroups of $\G$. Up to conjugation by $\G(\R)$, there exists a unique fundamental Cartan subgroup of $\G$ (see \cite[Proposition 6.61]{kn} for example).

\begin{defi}
An involutive algebraic automorphism $C: \G\rightarrow \G$ defined over $\R$ is  called a \textit{fundamental Chevalley involution} of $\G$ if  there exists a fundamental Cartan subgroup $\mathsf H$ of $\G$ such that $C(h)=h^{-1}$ for all $h\in \mathsf H(\mathbb{R})$.
\end{defi}

Let $C: \G\rightarrow \G$ be a fundamental Chevalley involution. The following result, which asserts the existence and uniqueness of fundamental Chevalley involutions, is proved by Jeffrey Adams in \cite[Theorem 1.2]{ad}.

 \begin{theorem} \lb{g} There exists a fundamental Chevalley involution $C$ of $\G$. If $C'$ is another fundamental Chevalley involution of $\G$, then
 \[
  C'=\Ad_g\circ C\circ \Ad_{g^{-1}}
 \]
 for some $g\in \G(\R)$, where $\Ad_g: \G\rightarrow \G$ denotes the conjugation by $g$.
  \end{theorem}

Let $C: \G\rightarrow \G$ be a fundamental Chevalley involution.
In \cite[Theorem 1.2]{ad}, Adams also proves that
for every semisimple element $g \in \G(\R)$,
$C(g)$ and $g^{-1}$ are conjugate to each other in ${\G}(\R)$.
Following Adams (\cite{ad}), we will prove the following generalization of this result.

\begin{theorem}\label{main1}
For every element $g \in \G(\R)$,
$C(g)$ and $g^{-1}$ are conjugate to each other in ${\G}(\R)$.

\end{theorem}

\newcommand{\g}{{\mathfrak{g}}}

Let $\g$ denote the Lie algebra of $\G$. By differentiation, the  fundamental Chevalley involution $C:\G\rightarrow \G$  induces an involutive automorphism   $C:\g\rightarrow \g$.
We will also prove the following Lie algebra analogue of Theorem \ref{main1}.
\begin{theorem}\label{main2}
For every element $X \in \g$,
$C(X)$ and $-X$ are conjugate to each other by ${\G}(\R)$.

\end{theorem}

\section{Semisimple elements}

We continue with the notation of the Introduction.

Recall that an element $g\in \G(\R)$ is said to be semisimple (or unipotent) if for every algebraic representation of $\G$ on every finite dimensional real vector space $V$, $g$ acts on $V$ by a semisimple (or unipotent) linear operator.
An element $X\in \g$ is said to be semisimple (or nilpotent) if for every aforementioned representation $V$, $X$ acts on $V$ by a semisimple (or nilpotent) operator.

As we already mentioned, the following lemma is proved in \cite[Theorem 1.2]{ad}.
\begin{lemma}\label{lemsg0}
Theorem \ref{main1} holds for all semisimple elements $g \in \G(\R)$.
\end{lemma}
\newcommand{\T}{{\mathsf T}}

The following  lemma is well-known.
 \ble \lb{x}Let $\mathsf T$ be an algebraic torus defined over $\mathbb{R}$. Then there is some $t\in\mathsf T(\mathbb{R})$ such that the cyclic group $\langle t\rangle $ generated by $t$ is Zariski-dense in $\mathsf T(\mathbb{C})$.

 \ele
 \bp
We sketch a proof for the convenience of the reader. Note that there are only countably many proper algebraic subgroups of $\mathsf T$ that are defined over $\R$. Thus the set
\[
  \bigcup_{\mathsf S\textrm{ is a proper algebraic subgroup of $\mathsf T$  defined over $\R$}} \mathsf S(\R)
\]
is a proper subset of $\T(\R)$. Pick an element  $t$ in the complementary set, and the lemma follows.

 \ep

\newcommand{\C}{{\mathbb C}}
\begin{lemma}\label{lemsg1}
Theorem \ref{main2} holds for all semisimple elements $X\in \g$.
\end{lemma}
\begin{proof}
Suppose that $X\in \g$ is semisimple. Then there is a Cartan subalgebra $\mathfrak h$ of $\g$ that contains $X$. Write $\mathsf H$ for the centralizer of $\mathfrak h$, which is a Cartan subgroup of $\G$ defined over $\R$. By Lemma \ref{x}, there is an element $t\in \mathsf H(\R)$ such that the cyclic group $\langle t\rangle$ is Zariski dense in $\mathsf H(\mathbb C)$.
By Lemma \ref{lemsg0}, there is an element
$g_0\in \G(\R)$ such that
\[
  C(t)=\Ad_{g_0}(t^{-1}).
\]
The Zariski dense property implies that
\[
  C(h)=\Ad_{g_0}(h^{-1})\qquad \textrm{for all }h\in \mathsf H(\C).
\]
By taking differential, we know that
\[
  C(Y)=\Ad_{g_0}(-Y)\qquad \textrm{for all }Y\in \mathfrak h.
\]
This proves the lemma.

\end{proof}

\section{Unipotent elements}
Write $\mathcal C(\g)$ for the set of $\G(\R)$-orbits in $\g$ under the adjoint action. Write $\mathcal C_{ss}(\g)$ and $\mathcal C_{nil}(\g)$ for the subsets of $\mathcal C(\g)$ consisting of the semisimple orbits and the nilpotent orbits, respectively. Define a map
\[
\begin{array}{rcl}
  \iota:\mathcal C_{nil}(\g)&\rightarrow & \mathcal C_{ss}(\g),\\
  \textrm{the class of $E\in \g$}&\mapsto &\textrm{the class of $E-F\in \g$},
  \end{array}
\]
where $E$ is a nilpotent element in $\g$ such that $(E, [E,F], F)$ forms an $\mathfrak s\mathfrak l_2$-triple. This map is well-defined and injective (\cite[Theorem 9.2.3, Theorem 9.4.6]{cm}).

The involutive automorphism $C:\g\rt \g$ descends to a map
\[
C:\mathcal C(\g)\rt \mathcal C(\g).
\]
By restriction one has the maps $C:\mathcal C_{ss}(\g)\rt \mathcal C_{ss}(\g)$ and  $C:\mathcal C_{nil}(\g)\rt \mathcal C_{nil}(\g)$.

The involutive map $J:\g\rt \g,\ X\mapsto -X$ also descends to a map
\[
J:\mathcal C(\g)\rt \mathcal C(\g).
\]
By restriction one also has the maps $J:\mathcal C_{ss}(\g)\rt \mathcal C_{ss}(\g)$ and  $J:\mathcal C_{nil}(\g)\rt \mathcal C_{nil}(\g)$.

 \begin{lemma}\label{lemsg5}
Theorem \ref{main2} holds for all nilpotent elements $X \in \g$.
\end{lemma}
\bp
It is directly verified that $C$ and $J$ both commutes with $\iota:\mathcal C_{nil}(\g)\rightarrow \mathcal C_{ss}(\g)$. As $\iota$ is injective and $C=J$ on $\mathcal C_{ss}(\g)$ by Lemma \ref{lemsg1}, one has $C=J$ on $\mathcal C_{nil}(\g)$ and the lemma follows.
\ep

 \begin{lemma}\label{lemsg6}
Theorem \ref{main1} holds for all unipotent elements $g \in \G(\R)$.
\end{lemma}
\bp
Suppose that  $g\in G(\mathbb{R})$ is a unipotent element. Then there is a unique nilpotent element $X\in \g$ such that  $\exp(X)=g$. By Lemma \ref{lemsg5},
\[
 C(X)=\Ad_k(-X)\qquad  \textrm{for  some $k\in\G(\R)$}.
\]
Thus
\[
C(g)=C(\exp (X))=\exp(C(X))=\exp(\Ad_k(-X))=\Ad_k(\exp(-X))=\Ad_k(g^{-1}).
\]
This proves the lemma.

\ep

\section{General elements}

Let $s\in \G(\R)$ be a semisimple element. Write $\mathsf Z$ for the centralizer of $s$ in $\G$, which is a reductive linear algebraic group defined over $\R$. Denote by $\mathsf Z^\circ$ the identity connected component of $\mathsf Z$, which is a connected reductive linear algebraic group defined over $\R$. Let $\mathsf T$ be a fundamental Cartan subgroup of $\mathsf Z^\circ$.

\begin{lemma}\label{exg0}
The element $s$ belongs to $\mathsf T(\R)$.
\end{lemma}
\begin{proof}
Note that $s\in \mathsf Z(\C)$. As $s$ is semisimple, it lies in some Cartan subgroup of $\mathsf G(\mathbb C)$ (see  \cite[Theorem 6.4..5]{sp}). Thus $s\in \mathsf Z^\circ(\C)$, and hence $s\in \mathsf Z^\circ(\R)$. Moreover, $s$ is in the center of  $\mathsf Z^\circ$, and hence   $s\in \mathsf T(\R)$.
\end{proof}

\begin{lemma}\label{exg}
There exists an element $g\in \G(\R)$ such that
\[
 C(x)= \Ad_g(x^{-1})\qquad \textrm{for all $x\in \mathsf T(\R)$}.
\]
\end{lemma}
\begin{proof}
Let $t\in \mathsf T(\R)$ be as in Lemma \ref{x}. By Lemma \ref{lemsg0}, there exists an element $g\in \G(\R)$ such that
\[
   C(t)= \Ad_g(t^{-1}).
\]
This clearly implies the lemma.
\end{proof}

Let $g\in \G(\R)$ be as in Lemma \ref{exg}.

\begin{lemma}
The automorphism
\[
  \Ad_{g^{-1}}\circ C: \G\rightarrow \G
\]
stabilizes the algebraic subgroup $\mathsf Z^\circ$.
\end{lemma}
\bp
Lemma \ref{exg0} implies that
\be\label{adg}
   \left(\Ad_{g^{-1}}\circ C\right)(s)=s^{-1}.
   \ee This clearly implies the lemma.
\ep

Let $C': \mathsf Z^\circ\rightarrow \mathsf Z^\circ$ be a fundamental Chevalley involution.

\begin{lemma}\label{k}
There is an element $k\in \mathsf Z^\circ(\R)$ such that
\[
  ( \Ad_{g^{-1}}\circ C)|_{\mathsf Z^\circ}= \Ad_{k^{-1}}\circ C':  \mathsf Z^\circ\rightarrow \mathsf Z^\circ.
\]
\end{lemma}

\begin{proof}
By the uniqueness of the fundamental Cartan subgroups, there exists an element $k_0\in \mathsf Z^\circ(\R)$ such that
\[
  C'(\Ad_{k_0}(x))=\Ad_{k_0}(x^{-1})\qquad \textrm{for all }x\in \mathsf T(\R).
\]
Then the automorphism
\[
  ( \Ad_{g^{-1}}\circ C)\circ (\Ad_{k_0^{-1}}\circ C'\circ \Ad_{k_0}): \mathsf Z^\circ\rightarrow \mathsf Z^\circ
\]
fixes all elements of $\mathsf T(\R)$. By \cite[Lemma 3.4]{ad}, this automorphism equals $\Ad_t:  \mathsf Z^\circ\rightarrow \mathsf Z^\circ$ for some $t\in \mathsf T(\R)$. Hence
\begin{eqnarray*}
&&  ( \Ad_{g^{-1}}\circ C)|_{\mathsf Z^\circ}\\
  &=&\Ad_t\circ (\Ad_{k_0^{-1}}\circ C'\circ \Ad_{k_0})\\
  &=&\Ad_t\circ (\Ad_{k_0^{-1}}\circ \Ad_{C'(k_0)}\circ C')\\
  &=&\Ad_{t \cdot k_0^{-1}\cdot C'(k_0)}\circ C'.
\end{eqnarray*}
This proves the lemma.
\end{proof}

\begin{lemma}\label{general}
For every unipotent element $u\in \G(\R)$ that commutes with $s$, there exists an element $g_0\in \G(\R)$ such that
\[
C(su)=\Ad_{g_0}(s^{-1} u^{-1}).
\]
\end{lemma}
\begin{proof}
Note that $u\in \mathsf Z^\circ(\R)$. By Lemma \ref{lemsg6}, there is an element $k'\in \mathsf Z^\circ(\R)$ such that
\[
  C'(u)=\Ad_{k'}(u^{-1}).
  \]
  Let $k\in \mathsf Z^\circ(\R)$ be as in Lemma \ref{k}.
  Then we have that
  \begin{eqnarray*}
&&  ( \Ad_{g^{-1}}\circ C)(su)\\
  &=& s^{-1} \cdot \left(\Ad_{k^{-1}}\circ C'\right)(u) \qquad \textrm{(by \eqref{adg} and Lemma \ref{k})}\\
  &=& s^{-1} \cdot \Ad_{k^{-1}k'}(u^{-1})\\
    &=&   \Ad_{k^{-1}k'}(s^{-1} u^{-1})
  \end{eqnarray*}
This proves the lemma.
\end{proof}

By using the Jordan decomposition (see \cite[Theorem 9.2]{w}), Lemma \ref{general} implies Theorem \ref{main1}. The same method also proves Theorem \ref{main2}.

\section*{Acknowledgements}
The problem of this note was posed by Jeffrey Adams, and the proof was a result of a discussion between  Jeffrey Adams, David Vogan and   Binyong Sun during the conference  ``Branching Laws" that was held in 2012 in the Institute for Mathematical Sciences, National University of Singapore. The authors are grateful to  Adams and Vogan for the  very helpful discussions. They also thank the Institute for Mathematical Sciences at NUS for the hospitality. G. Han was supported by  National Natural Science Foundation of China (No. 12071422) and  Zhejiang Province Science Foundation of China (No. LY14A010018).
B. Sun was supported by National Key $\textrm{R}\,\&\,\textrm{D}$ Program of China (No. 2020YFA0712600) and the National Natural Science Foundation of China (No. 11688101).

\end{document}